\documentclass[reqno]{amsart}
\usepackage{amsmath,amsthm,amssymb}
\usepackage{latexsym}
\usepackage{eucal}

\newtheorem{theorem}{Theorem}[section]
\newtheorem{lemma}{Lemma}[section]

\theoremstyle{definition}

\newtheorem{open problem}{Open Problem}

\def\cal{\mathcal}

\begin{document}
\title[Weak asymptotics for Schr\"odinger evolution
 ]{Weak asymptotics for Schr\"odinger evolution}
\author{Sergey A. Denisov}
\address{
\begin{flushleft}
University of Wisconsin--Madison\\  Mathematics Department\\
480 Lincoln Dr., Madison, WI, 53706, USA\\
  denissov@math.wisc.edu
\end{flushleft}
  }
 \maketitle
\begin{abstract}
In this short note, we apply technique developed in \cite{den1} to
study the long-time evolution for Schr\"odinger equation with slowly
decaying potential.
\end{abstract} \vspace{1cm}

Consider
\[
H=-\partial^2_{xx}+q, \quad x>0
\]
with Dirichlet boundary condition at zero. If $q=0$, we denote the
operator by $H_0$. Through the paper, the potentials is real valued
and satisfies the following condition
\begin{equation}
q(x)\ln(|x|+2)\in L^2(\mathbb{R}^+) \label{ops1}
\end{equation}

 We will use the asymptotics of generalized
 eigenfunctions obtained in \cite{den1} to prove
 existence of modified wave operators.  Unfortunately,
 the limits will be understood in some averaged sense only.
 In the meantime, the methods are rather robust and
 can be used for other dispersive equations. For $q\in
 L^p(\mathbb{R}^+),\, 1\leq p<2$ the existence of modified wave operators was
 proved in \cite{ck}.\bigskip

We will start with some definitions. Assume that $f(x)\in
L^2(\mathbb{R}^+)$ and take its odd continuation to $\mathbb{R}$.
Call it $f_{o}(x)$. Then
\begin{equation}
e^{it\partial^2_{x}}f_o\sim \kappa
\frac{e^{ix^2/(4t)}}{\sqrt{t}}\hat{f}_o(x/(2t))  \quad {\rm in
\,\,\, }L^2(\mathbb{R}), \quad t\to\infty\label{ezh}
\end{equation}
where
\[
\kappa=-\frac{1}{(1+i)\sqrt{2\pi}}, \quad \hat{f_o}(\omega)=\int
f_o(x)e^{i\omega x}dx
\]
so $\hat{f}_o$ denotes the inverse Fourier transform. (In this
paper, $f\sim g$ as $t\to\infty$ if $\|f-g\|\to 0$ as $t\to\infty$
in the specified metric). The asymptotics (\ref{ezh}) is easy to
check if $\hat{f}_o$ is infinitely smooth and compactly supported
away from zero. The general $L^2$ case then follows upon making the
simple observation that the l.h.s. and the r.h.s. are unitary in $f$
and then using the approximation argument.

Then, by symmetry,
\[
e^{-iH_0t}f\sim
\kappa\frac{e^{ix^2/(4t)}}{\sqrt{t}}\hat{f}_o(x/(2t))\chi_{x>0},
\quad {\rm in}\,\, L^2(\mathbb{R}^+), \quad t\to+\infty
\]
We will need to modify the free evolution. The modification will be
made in the physical space as follows
\[
U(t)f=\kappa \frac{e^{ix^2/(4t)}}{\sqrt{t}}\hat{f}_o(x/(2t))
\exp\left(-i\frac{t}{x}\int_0^x q(s)ds\right)
\]

Let $u(x,k)$ be the solution to the Cauchy problem
\[
-u''+qu=k^2u, \quad u(0,k)=0, \quad u'(0,k)=1, \quad E=k^2
\]
and the spectral measure
\[
d\rho_E=d\rho_s(E)+\mu(E)dE
\]
The negative eigenvalues, if there are any, will be denoted by
$\{-\kappa_j^2\}, j=1,2,\ldots$. We will need the following trivial
lemma
\begin{lemma}
Assume that $f(x)$ is infinitely smooth function with compact
support, then
\[
\left|\frac{1}{\sqrt{t}} \int_x^\infty
f(st^{-1})\exp\left(i\left(\frac{s^2}{2t}-sk\right)\right)ds\right|\lesssim
\frac{\sqrt{t}}{x-kt+\sqrt{t}}
\]
for any $x>kt$ and for $t>1, k>0$.\label{st}
\end{lemma}
\begin{proof}
After the change of variables $s=t(k+u)$, we are left with
\[
e^{-itk^2/2}\int_{xt^{-1}-k}^\infty f(k+u) \sqrt{t}e^{itu^2/2}du
\]
Integration by parts and the simple estimate
\[
\left|\int_x^\infty e^{iu^2}du\right|\lesssim \frac{1}{1+x}, \quad
x>0
\]
finish the proof.
\end{proof}

Below, we will use some notations from \cite{den1}. The main result
of the paper is
\begin{theorem}
For any $f\in L^2(\mathbb{R}^+)$, there is $W_f$ such that
\[
\frac{1}{T}\int_0^T \|e^{itH}U(t)f-W_f\|^2_{L^2(\mathbb{R}^+)}dt\to
0, \quad T\to\infty
\]
The exact expression for $W_f(x)$ will be given in the proof.
\end{theorem}
Thus, for most large $t$, we have $e^{-itH}W_f\sim U(t)f$ in
$L^2(\mathbb{R}^+)$.

\begin{proof}
It is sufficient to show that
\[
\frac{1}{T}\int_T^{2T}
\|e^{itH}U(t)f-W_f\|^2_{L^2(\mathbb{R}^+)}dt\to 0, \quad T\to\infty
\]
(this follows, e.g., from the diadic decomposition argument).
\bigskip

Assume that $f$ is such that $\hat{f}_o(k)\chi_{k>0}$ is infinitely
smooth with compact support on, say,  $[a,b]\subset \mathbb{R}^+$.
If we prove the convergence in this case, then the standard
approximation argument can handle the general case. Denote the
generalized Fourier transform of $U(t)f$ by $\breve{\psi}(t,k)$. For
$k\geq 0$,
\[
\breve{\psi}(t,k)=\int_0^\infty u(x,k)[U(t)f](x)dx
\]
Using the formula (9) from \cite{den1}, we have
\[
\breve\psi(t,k)=-I_1+I_2
\]
where
\begin{equation}
I_1=\frac{1}{2ik}\int_{2at}^{2bt}
[U(t)f](x){j_m}(k,x)e^{-ikx+i\phi(0,k,x)}dx\label{ij}
\end{equation}

\begin{equation}
I_2=\frac{1}{2ik}\int_{2at}^{2bt}
[U(t)f](x)\overline{j_m}(k,x)e^{ikx-i\phi(0,k,x)}dx \label{ij1}
\end{equation}
Splitting the integration in $I_1$ as
\begin{equation}
\int_{2at}^{2bt} =\int_{2at}^{2kt}+\int_{2kt}^{2bt}=J_1+J_2
\label{jaj}
\end{equation}
we have
\begin{eqnarray}
J_1=-\frac{\kappa}{2ik\sqrt{t}}\int_{2kt}^{2bt} {j_m}(k,x)
\exp\left(i\left(\frac{1}{2k}-\frac{t}{x}\right)\int_0^x q(s)ds\right)\cdot \hspace{2cm}\nonumber\\
\hspace{1cm}\partial_x\left(\int_x^{2bt} \hat{f}_o(s/(2t))
\exp\left(i\left(\frac{s^2}{4t}-sk\right)\right)ds\right)dx
\end{eqnarray}

Integrating by parts,
\[
J_1=L_1+L_2
\]
where
\[
L_1=\frac{\kappa}{2ik} j_m(k,2kt)\left(\frac{1}{\sqrt
{t}}\int_{2kt}^{2bt}
\hat{f}_o(s/(2t))\exp\left(i\left(\frac{s^2}{4t}-sk\right)\right)ds\right)
\]
The stationary phase argument gives
\[
\frac{e^{itk^2}}{\sqrt {t}}\int_{2kt}^{2bt}
\hat{f}_o(s/(2t))\exp\left(i\left(\frac{s^2}{4t}-sk\right)\right)ds\to
\hat{f}_o(k)\left(\sqrt{\pi/2}(1+i)\right)
\]
uniformly in $k$ over any compact. So, for any fixed $\delta, M>0$
and interval $I=[\delta,M]$,
\begin{eqnarray}
\limsup_{T\to\infty} \frac{1}{T}\int_0^T \int_I
\left|e^{itk^2}L_1(k,t)-\frac{\kappa}{2ik}\hat{f}_o(k)j_m(k)
\chi_{\Theta}(E)\sqrt{\pi/2}(1+i)\right|^2d\rho(E)\lesssim \hspace{1cm}\nonumber\\
\limsup_{T\to\infty}\frac{1}{T}\int_I d\rho(E) \int_0^T
\left|\hat{f}_o(k)(j_m(k,2kt)-j_m(k)\chi_{\Theta}(E))\right|^2dt\label{vot}
\end{eqnarray}
where $\Theta=\mathbb{R}^+\backslash \Theta_s$, $\Theta_s$ is the
support of $d\rho_s(E)$. In the argument above we also used the
uniform bound
\[
\limsup_{T\to\infty} \frac{1}{T}\int_0^T dt \int_I d\rho(E)
|j_m(k,2kt)|^2<\infty
\]
which follows from the estimate (31) in \cite{den1} after change of
variables $2kt=\tau$.

 Making the change of variables $kt=t_1$ in (\ref{vot}),
we have
\begin{eqnarray*}
\limsup_{T\to\infty}\frac{1}{T}\int_I d\rho(E) \int_0^T
\left|e^{itk^2}L_1(k,t)-\frac{\kappa}{2ik}\hat{f}_o(k)j_m(k)\chi_{\Theta}(E)\sqrt{\pi/2}(1+i)\right|^2dt
\lesssim\\
\limsup_{T\to\infty}\frac{1}{T}\int_I d\rho(E)\int_0^{C_1T}
|j_m(k,2t_1)-j_m(k)\chi_{\Theta}(E)|^2dt_1
\end{eqnarray*}
The last limit is equal to zero by the theorem 3.1 in \cite{den1}.

For $L_2$, using the formula for derivative of $\partial_x j(k,x)$
(\cite{den1}, Lemma 2.1), we get
\begin{eqnarray*}
|L_2|\lesssim \frac{1}{\sqrt t} \int_{2kt}^{2bt} |j_m(k,x)|
\left(|q(x)|+\frac{t}{x^2}\int_0^x |q(s)|ds\right)
\\\hspace{2cm}
\left(\int_x^{\infty} \hat{f}_o(s/(2t))
\exp\left(i\left(\frac{s^2}{4t}-sk\right)\right)ds\right)dx
\end{eqnarray*}
for $k\in I$. By lemma \ref{st}  we have an estimate
\[
|L_2| \lesssim  \int_{2kt}^{2bt} |j_m(k,x)|
\left(|q(x)|+\frac{1}{x}\int_0^x |q(s)|ds\right) \frac{\sqrt{t}}{(x-2kt)+\sqrt t}dx
\]
Making the change of variables $kt=t_1$ once again, we have
\[
\frac{1}{T}\int_T^{2T} \int_I |L_2(k,t)|^2d\rho(E)dt\lesssim M_1+M_2
\]
where
\[
M_1=\int_{C_1T}^{C_2T}\int_I\left(\int_{2t_1}^{C_3t_1}
\frac{|j_m(k,x)q(x)|}{(x-2t_1)+\sqrt{T}}dx\right)^2d\rho(E)dt_1
\]
\[
M_2=\int_{C_1T}^{C_2T}\int_I\left(\int_{2t_1}^{C_3t_1}
\frac{|j_m(k,x)q_1(x)|}{(x-2t_1)+\sqrt{T}}dx\right)^2d\rho(E)dt_1
\]
where
\[
q_1(x)=\frac{1}{x}\int_0^x |q(u)|du
\]
By Young inequality for convolutions and the following estimate
(\cite{den1}, (31))
\[
\sup_x\int_I |j_m(x,k)|^2d\rho(E)<\infty
\]
we have
\begin{eqnarray*}
M_1\lesssim \int_I d\rho(E) \left(\int_{C_1T}^{C_2T} |j_m(k,x)|^2
q^2(x)dx\right) \left( \int_0^{C_3T} \frac{dx}{x+\sqrt{T}}\right)^2\\
\lesssim \ln ^2T \int_{C_1T}^{C_2T} q^2(x)dx\lesssim
\int_{C_1T}^\infty q^2(x)\ln^2 x dx\to 0, \quad T\to\infty
\end{eqnarray*}
For $M_2$, the estimate is similar
\begin{eqnarray*}
M_2\lesssim \ln^2 T \int_{C_1T}^{C_2T} q_1^2(x)dx\lesssim
\frac{\ln^2T}{T}\left(\int_0^{CT} |q(u)|du\right)^2\hspace{5cm}\\
\lesssim \bar{o}(1)+ \frac{\ln^2T}{T}\left(\int_{\sqrt{T}}^{CT}
q^2(x) \ln^2 (2+x)dx\right)\left(\int_{\sqrt{T}}^T
\ln^{-2}(2+x)dx\right)\to 0, \quad T\to\infty
\end{eqnarray*}
The term $J_2$ in (\ref{jaj}) can be handled similarly and we have
\begin{eqnarray*}
\frac{1}{T}\int_I d\rho(E) \int_T^{2T}
|e^{itk^2}J_{1(2)}(k,t)-\frac{\kappa}{2ik}\hat{f}_o(k)j_m(k)\chi_{\Theta}(E)\sqrt{\pi/2}(1+i)|^2dt\to
0, \quad T\to\infty
\end{eqnarray*}

\bigskip

For $I_2$, the analysis is identical with the exception that
integration by parts gives
\[
\frac{1}{T} \int_I d\rho(E) \int_T^{2T} |I_2|^2dt\to 0
\]
Thus, for any $\delta,M>0$,
\[
\frac{1}{T}\int_T^{2T} \int_\delta^M
|\breve\psi(t,k)+\frac{\kappa}{2ik}\hat{f}_o(k)j_m(k)\chi_{\Theta}(E)\sqrt{2\pi}(1+i)|^2d\rho(E)\to
0
\]
In the statement of the theorem, we can choose $W_f(x)$ as the
function with generalized Fourier transform equal to
\[
\breve{W}_f(E)=-\frac{\kappa}{2ik}\hat{f}_o(k)j_m(k)\chi_{\Theta}(E)\sqrt{2\pi}(1+i)
\]
\bigskip

The function $\hat{f}_o(k)$ is infinitely smooth with compact
support and $U(t)f$ travels ballistically. We also have
$\sup_{t}\|U(t)f\|_{W^{1,2}(\mathbb{R})}<\infty$ and so by taking
suitable cutoff near the origin
\[
U(t)f=s_1(t)+s_2(t)
\]
where $\|s_1(t)\|_2\to 0$ and $s_2(t)\in \cal{D}(|H|^{1/2})$ with
\[
\sup_t\int (k^2+1)|\breve s_2(t,k)|^2d\rho(E)<\infty
\]
since
\[
(Hs_2,s_2)=\int_0^\infty (|s_2'|^2+q|s_2|^2)dx
\]

 Therefore,
\[
\limsup_{t\to\infty}  \int_M^\infty |\breve\psi(t,k)|^2d\rho(E)\to 0
\]
as $M\to\infty$. We are left with proving
\[
\limsup_{t\to\infty} \|P_{(-\infty,\delta]}U(t)f\|_2\to 0
\]
as $\delta\to 0$. If $e_j(x)$ is eigenfunction for negative
eigenvalue $-\kappa_j^2$ and $\|e_j\|_{L^2(\mathbb{R}^+)}=1$, then
\[
\langle U(t)f,e_j\rangle \to 0, \quad t\to\infty
\]
for each fixed $j$ so we just have to show
 that $L^2$ norm of $U(t)f$ can not accumulate near zero energy,
e.g., that
\begin{equation}
\limsup_{t\to\infty}\|P_{[-\delta,\delta]}U(t)f\|_2\to 0, \quad
\delta\to 0 \label{orig}
\end{equation}

 Write
$q(x)\chi_{x>0}=q_1(x)+q_2(x)$, where
\[
\hat{q}_1(\omega)=\hat{q}(\omega)\chi_{|\omega|<1}, \quad q_2=q-q_1
\]
Clearly, $q_1\in W^{\infty,2}(\mathbb{R}^+)$ and $q_2=v'$ where
$v\in W^{1,2}(\mathbb{R})$. The multiplicative correction in $U(t)$
is
\[
\exp\left(-i\frac{t}{x}\int_0^x
q(s)ds\right)=\exp\left(-i\frac{t}{x}\int_0^x
q_1(s)ds\right)\exp\left(-i\frac{t}{x}(v(x)-v(0))\right)
\]

We have
\[
\left(\exp\left(-itx^{-1}v(x)\right)-1\right)\frac{e^{ix^2/(4t)}}{t^{1/2}}\hat{f}_o(x/(2t))\to
0
\]
in $L^2(\mathbb{R})$ as $t\to\infty$ since
$\lim_{x\to\infty}v(x)=0$. We can write
\begin{eqnarray}
\exp\left(-i\frac{t}{x}\left(C+\int_0^x
q_1(s)ds\right)\right)\frac{e^{ix^2/(4t)}}{\sqrt{t}}\hat{f}_o(x/(2t))\hspace{2cm}
\\\hspace{2cm}
\sim C_1\exp\left(-i\frac{t}{x}\left(C+\int_0^x
q_1(s)ds\right)\right)\cal{F}\left( e^{-itw^2}
\hat{f}_o(\omega)\right)
\end{eqnarray}
where $\cal{F}$ is Fourier transform. Since $\hat{f}_o$ has support
away from the zero,
\[
\cal{F}\left( e^{-itw^2} \hat{f}_o(\omega)\right)=\psi'', \quad
\psi= -\cal{F}\left( w^{-2}e^{-itw^2} \hat{f}_o(\omega)\right)
\]
and $\psi$ travels ballistically in time as well. Therefore, we can
write
\[
\exp\left(-i\frac{t}{x}\left(C+\int_0^x
q_1(s)ds\right)\right)\cal{F}\left( e^{-itw^2}
\hat{f}_o(\omega)\right)=s_1''+s_2, \quad x>1
\]
where
\[
s_1= \exp\left(-i\frac{t}{x}\left(C+\int_0^x
q_1(s)ds\right)\right)\psi
\]
Since $s_1$ travels ballistically,
$\sup_t\|s_1\|_{W^{2,2}(\mathbb{R})}<\infty$, and  $\|s_2\|_2\to 0$
as $t\to\infty$, we can write
\[
\exp\left(-i\frac{t}{x}\left(C+\int_0^x
q_1(s)ds\right)\right)\cal{F}\left( e^{-itw^2}
\hat{f}_o(\omega)\right)=Hl_1+l_2
\]
where $l_1(0)=0,$ $\sup_t\|l_1\|_{W^{2,2}(\mathbb{R}^+)}<\infty$,
$\|l_2\|_2\to \infty, \quad t\to\infty$. These representations shows
that
\[
\limsup_{t\to\infty}\int_{\delta}^\delta
|\breve{\psi}(t,k)|^2d\rho(E)\to 0,
\]
as $\delta\to 0$ (i.e. (\ref{orig}) holds). So,
\[
\frac{1}{T}  \int_0^T dt \int_{\mathbb{R}}
|\breve\psi(t,k)-\breve{W}_f(E)|^2d\rho(E)\to 0
\]

and the theorem is proved.
\end{proof}

It is an interesting problem to try to relax (\ref{ops1}) to just
$q\in L^2(\mathbb{R}^+)$. We are not able to do that at this moment.

\bigskip\bigskip{\bf Acknowledgements.} This research was supported by Alfred
P. Sloan Research Fellowship and NSF Grant DMS-0758239.

\end{document}